\documentclass[11pt, reqno]{amsart}
\usepackage{lineno,hyperref}
\usepackage{cases}
\usepackage[square, comma, sort&compress, numbers]{natbib}%?¦Ï????????????
\usepackage{float}
\usepackage{pifont}
\usepackage{bbding}
\usepackage{amssymb}
\usepackage{mathrsfs}
\usepackage{subfigure}
\usepackage{caption}
\usepackage{geometry}
\usepackage{amsmath, amsthm, amscd, amsfonts, amssymb, graphicx, color}
\usepackage{lineno}
\newtheorem{thm}{Theorem}
\newtheorem{conj}{Conjecture}
\newtheorem{theorem}{Theorem}[section]
\newtheorem{lemma}[theorem]{Lemma}

\newtheorem{corollary}[theorem]{Corollary}
\theoremstyle{definition}
\newtheorem{definition}{Definition}

\newtheorem{remark}{Remark}

\newtheorem{case}{Case}
\theoremstyle{remark}
\numberwithin{equation}{section}
\DeclareMathOperator{\RE}{Re}

\newcommand{\ID}{\mathbb{D}}

\begin{document}
\thispagestyle{empty} \setcounter{page}{1}

%\noindent\parbox{2.85cm}{\includegraphics*[keepaspectratio=true,scale=1.75]{BJMA.jpg}}
%\noindent\parbox{4.85in}{\hspace{0.1mm}\\[1.5cm]\noindent
%Banach J. Math. Anal. 0 (0000), no. 0, 00--00\\
%$\frac{\rule{4.55in}{0.05in}}{{}}$\\
%{\footnotesize
%\textcolor[rgb]{0.65,0.00,0.95}{\textsc{\textbf{\large{B}}anach
%\textbf{\large{J}}ournal of \textbf{\large{M}}athematical
%\textbf{\large{A}}analysis}}\\
%ISSN: 1735-8787 (electronic)\\
%\textcolor[rgb]{0.00,0.00,0.84}{\textbf{http://www.math-analysis.org }}\\
%$\frac{{}}{\rule{4.55in}{0.05in}}$}\\[.5in]}

\title[Radius of fully starlikeness and fully convexity of harmonic...]
{Radius of fully starlikeness and fully convexity of harmonic linear differential operator}
\author[Z. Liu]{ZhiHong Liu}
\address{Z. Liu, School of Mathematics and Econometrics, Hunan University,
Changsha 410082, Hunan, People's Republic of China.
\vskip.03in College of Mathematics, Honghe University,
Mengzi 661199, Yunnan, People's Republic of China.}
\email{\textcolor[rgb]{0.00,0.00,0.84}{liuzhihongmath@163.com}}

\author[S. Ponnusamy]{Saminathan Ponnusamy}
\address{S. Ponnusamy, Department of Mathematics,
Indian Institute of Technology Madras, Chennai--600 036, India.}
\email{\textcolor[rgb]{0.00,0.00,0.84}{samy@isichennai.res.in, samy@iitm.ac.in}}

%\dedicatory{This paper is dedicated to Professor ABCD}



\begin{abstract}
Let $f=h+\overline{g}$ be a normalized harmonic mapping in the unit disk $\ID$. In this paper, we obtain the sharp radius of univalence, fully starlikeness and fully convexity of the harmonic linear differential operators $D_f^{\epsilon}=zf_{z}-\epsilon\overline{z}f_{\overline{z}}~(|\epsilon|=1)$ and  $F_{\lambda}(z)=(1-\lambda)f+\lambda D_f^{\epsilon}~(0\leq\lambda\leq 1)$ when the coefficients of $h$ and $g$ satisfy harmonic Bieberbach coefficients
conjecture conditions. Similar problems are also solved when the coefficients of $h$ and $g$ satisfy the corresponding necessary conditions of the harmonic convex function $f=h+\overline{g}$. All results are sharp. Some of the results are motivated by the work of Kalaj et al. \cite{Kalaj2014} (Complex Var. Elliptic Equ. 59(4) (2014), 539--552).
\end{abstract}
\maketitle

%---------------------------------------------------------------------------------------------------------------%
\section{Introduction}
%---------------------------------------------------------------------------------------------------------------%
\noindent
Throughout this paper, let $\mathcal{H}$ denote the class of all complex-valued harmonic functions $f=h+\overline{g}$
defined on the unit disk $\ID=\{z\in \mathbb{C}:|z|<1\}$ normalized by $f(0)=0=h(0)=h'(0)-1$, where $h$ and $g$ are analytic in $\ID$ such that
\begin{equation}\label{phg}
h(z)=z+\sum_{n=2}^{\infty}a_{n}z^{n}\quad{\rm and }\quad g(z)=\sum_{n=1}^{\infty}b_{n}z^{n}.
\end{equation}
Also, let $\mathcal{H}_0=\{f=h+\overline{g}\in\mathcal{H}: g'(0)=0\}$ and $J_{f}(z)=|h'(z)|^2-|g'(z)|^2$ denote
the Jacobian of $f=h+\overline{g}$.
According to Lewy's theorem~(see~\cite{Lewy1936} or ~\cite{Duren2004}), $f\in\mathcal{H}$ is
locally univalent and sense-preserving if and only if $J_{f}(z)>0$ in $\ID$, i.e $|g'(z)|<|h'(z)|$;
or equivalently, the dilatation $\omega(z)=g'(z)/h'(z)$ is analytic in $\ID$ with $|\omega(z)|<1$ in $\ID$. Note that if $f\in\mathcal{H}_0$,
then it is clear that $\omega(0)=0$. Let $\mathcal{S}_{H}$ denote the class of univalent and sense-preserving functions
$f=h+\overline{g}\in \mathcal{H}$. The class $\mathcal{S}_{H}$ for which $g(z)\equiv 0$ reduces to the class $\mathcal{S}$ of normalized univalent analytic functions in $\ID$. This is fundamental in the study of univalent functions theory. The geometric subclasses of $\mathcal{S}_{H}$ consisting of
the {\em convex}, {\em starlike} and {\em close-to-convex} functions in $\ID$ are denoted by $\mathcal{C}_{H}$, $\mathcal{S}_{H}^{*}$ and $\mathcal{K}_{H}$, respectively.
It is a standard practice to let
$$\mathcal{S}_{H}^{0}:=\mathcal{S}_{H}\cap \mathcal{H}_{0},\quad \mathcal{C}_{H}^{0}:=\mathcal{C}_{H}\cap \mathcal{H}_{0},\quad \mathcal{S}_{H}^{0*}:=\mathcal{S}_{H}^{*}\cap \mathcal{H}_{0},\quad  \mathcal{K}_{H}^{0}:=\mathcal{K}_{H}\cap \mathcal{H}_{0},
$$

We now recall one of the results of  Clunie and Sheil-Small \cite{Clunie1984} which connects analytic close-to-convex functions with that of
harmonic close-to-convex functions.

\begin{thm}\label{thmClinie}
Suppose that $h$ and $g$ are analytic in $\ID$ with $|g'(0)|<|h'(0)|$ and $h+\epsilon g$ is close-to-convex
for each $\epsilon,|\epsilon|=1$, then $f=h+\overline{g}$ is close-to-convex in $\ID$.
\end{thm}

As an analog of Bieberbach Conjecture proved by de Branges~\cite{DeBranges1985} for functions in $\mathcal{S}$,
Clunie and Shell-Small~\cite{Clunie1984} proposed the following conjecture for functions in $\mathcal{S}_{H}^{0}$.
\begin{conj}\label{ConjSH}
Let $f=h+\overline{g}\in\mathcal{S}_{H}^{0}$, where the representation of $h$ and $g$ are given by~\eqref{phg} with $b_1=0$. Then
for each $n\geq 2$,
\begin{equation}\label{BBHJ}
|a_{n}|\leq \frac{1}{6}(2n+1)(n+1), \quad |b_{n}|\leq \frac{1}{6}(2n-1)(n-1)
 ~\mbox{ and }~\big| |a_{n}|-|b_{n}| \big |\leq n .
% \quad (n=2,3,\cdots)
\end{equation}
Equality holds for the harmonic Koebe function defined by
\begin{equation}\label{eqHC}
\begin{split}
K(z)&=H_{1}(z)+\overline{G_1(z)} =\frac{z-\frac{1}{2}z^2+\frac{1}{6}z^3}{(1-z)^3}+\overline{\frac{\frac{1}{2}z^2+\frac{1}{6}z^3}{(1-z)^3}}.
\end{split}
\end{equation}
\end{conj}

Conjecture~\ref{ConjSH} has been verified for certain special situations, \emph{eg.} for the family $S_{H}^{0*}$ and the family
of typically real functions in $\ID$ by Clunie and Sheil-Small~\cite{Clunie1984} himself, for the family of functions that are convex
in one direction by Sheil-Small~\cite{Sheil1990}, and for the close-to-convex family $\mathcal{K}_{H}^{0}$ by Wang, Liang and Zhang~\cite{Wang2001}. More about harmonic mappings may be obtained from the monograph of Duren~\cite{Duren2004} and the survey article of Ponnusamy and Rasila~\cite{Samy2013}. For a recent development on Conjecture~\ref{ConjSH} for functions in $S_{H}^{0}$, we refer to the recent article of Ponnusamy et al. \cite{Samy2017}.

Also in~\cite{Clunie1984}, Clunie and Sheil-Small proved that if $f=h+\overline{g}\in \mathcal{C}_{H}^{0}$, then
\begin{equation}\label{BBHJ2}
|a_{n}|\leq \frac{n+1}{2} ~\mbox{ and }~|b_{n}|\leq \frac{n-1}{2}
%\quad{\rm and}\quad \big||a_n|-|b_n|\big|\leq 1
\end{equation}
hold for all $n\geq 2$. Equality holds for the harmonic right half-plane mapping $L\in \mathcal{C}_{H}^{0}$ given by
\begin{equation}\label{eqRL}
\begin{split}
L(z)&=H_2(z)+\overline{G_2(z)} =\frac{z-\frac{1}{2}z^2}{(1-z)^2}+\overline{\frac{-\frac{1}{2}z^2}{(1-z)^2}}.
\end{split}
\end{equation}

Our present article is motivated by the recent work of Kalaj et al.~\cite{Kalaj2014} and also by the results known in the analytic case.
Accordingly, we are mainly concerned with certain class of (sense-preserving harmonic) functions $f=h+\overline{g}\in \mathcal{H}_{0}$ satisfying the condition~\eqref{BBHJ} or~\eqref{BBHJ2}.

\begin{definition}\label{DfF}
Let $\mathcal{F}\subset \mathcal{H}$, where $\mathcal{F}=\mathcal{S}_{H}$ or $\mathcal{S}_{H}^{0}$; $\mathcal{F}=\mathcal{S}_{H}^{*}$ or $\mathcal{S}_{H}^{0*}$; $\mathcal{F}=\mathcal{C}_{H}$ or $\mathcal{C}_{H}^{0}$; $\mathcal{F}=\mathcal{K}_{H}$ or $\mathcal{K}_{H}^{0}$ etc.. For $r\in(0,1)$, we say that $f=h+\overline{g}\in \mathcal{F}$ in $|z|<r$ if $f_{r}(z)=r^{-1}f(rz)\in \mathcal{F}$ in $|z|<1$ in the usual sense.
\end{definition}

For example, we say that $f\in S_{H}^{0*}$ in $|z|<r$ whenever $f_r\in \mathcal{S}_{H}^{0*}$ (in $|z|<1$). This convention will be followed throughout this article.

The class $\mathcal{R}$ given by
$$\mathcal{R}=\left\{h:~\mbox{$h$ is analytic in $\ID$},\, h(0)=0=h'(0)-1,\, \RE\left( h'(z)\right)>0, ~z\in\ID \right\}
$$
provides a univalent criterion. As a harmonic analog of it, Ponnusamy et al.~\cite{Samy2013CVEE}, considered the class
$$\mathcal{K}_{H}^{1}=\left\{f=h+\overline{g}\in\mathcal{H}: \RE \left(h'(z)\right)>|g'(z)|,~ z\in \ID\right\}.
$$
We say that $f\in \mathcal{F} \mathcal{S}_{H}^{*}$ if $f\in \mathcal{S}_{H}$ maps each subdisk $|z|<r$, $r\in(0,1)$ onto a
domain which is starlike with respect to the origin. Functions in $\mathcal{F} \mathcal{S}_{H}^{*}$ is then called a fully starlike in $\ID$.
We may set $\mathcal{F} \mathcal{S}_{H}^{0*}:= \mathcal{F}\mathcal{S}_{H}^{*}\cap  \mathcal{H}_0$ for obvious reason.  Similarly, one can
define the  class of $\mathcal{F} \mathcal{C}_{H}$ of fully convex functions, and $\mathcal{F} \mathcal{C}_{H}^0:= \mathcal{F}\mathcal{C}_{H}\cap  \mathcal{H}_0$. See~\cite{Chuaqui2004} (also \cite{Kalaj2014}) for a detailed discussion on fully starlike and fully convex functions.

Moreover, in~\cite{Kalaj2014}, it was shown that $f=h+\overline{g}\in \mathcal{F} \mathcal{S}_{H}^{0*}\cap \mathcal{K}_{H}^{2}$ whenever the coefficients of $h$ and $g$ given by \eqref{phg} satisfy the condition
$$\sum_{n=2}^{\infty}n(|a_n|+|b_n|)\leq 1-|b_1|,
$$
where
$$\mathcal{K}_{H}^{2}=\left\{f=h+\overline{g}\in\mathcal{H}: |h'(z)-1|<1-|g'(z)|,~ z\in\ID\right\}.
$$
Obviously, we may let $\mathcal{K}_{H}^{1,0}=\mathcal{K}_{H}^{1}\cap \mathcal{H}_{0}$ and $\mathcal{K}_{H}^{2,0}=\mathcal{K}_{H}^{2}\cap \mathcal{H}_{0}$.

\begin{definition}\label{dfFC}
Function $f\in \mathcal{H}$ is said to be {\em fully convex of
order $\alpha~(0\leq\alpha< 1)$ } if it maps every circle $|z| = r < 1$ in a one-to-one manner onto a
convex curve satisfying
$$\frac{\partial}{\partial\theta}\left(\arg\left(\frac{\partial}{\partial\theta}
f(re^{i\theta})\right)\right)>\alpha, \quad 0\leq\theta<2\pi,0<r<1.
$$
If $\alpha = 0$, then the harmonic mapping $f$ is said to be {\em fully convex} (univalent) in $\ID$.
\end{definition}

\begin{definition}\label{dfFS}
A function $f\in \mathcal{H}$ is said to be {\em fully starlike of
order $\alpha~(0\leq\alpha< 1)$} if it maps every circle $|z| = r < 1$ in a one-to-one manner onto a
 curve satisfying
$$\frac{\partial}{\partial\theta}\left(\arg\left(
f(re^{i\theta})\right)\right)>\alpha, \quad 0\leq\theta<2\pi,0<r<1.
$$
If $\alpha = 0$, then the harmonic mapping $f$ is said to be {\em fully starlike} (univalent) in $\ID$.
\end{definition}

Let $\mathcal{FS}^{*}_{H}(\alpha)$ and $\mathcal{FC}_{H}(\alpha)$, respectively, denote the subclasses of $\mathcal{S}_{H}$ consisting of fully
harmonic starlike and fully harmonic convex functions of order $\alpha$. Also, set
$$\mathcal{F}\mathcal{S}_{H}^{0*}(\alpha)=\mathcal{F}\mathcal{S}_{H}^{*}(\alpha)\cap \mathcal{H}_0 ~\mbox{ and }~\mathcal{F}\mathcal{C}_{H}^{0}(\alpha)=\mathcal{F}\mathcal{C}_{H}(\alpha)\cap \mathcal{H}_0
$$
so that $\mathcal{F}\mathcal{S}_{H}^{0*}(0)=:\mathcal{F}\mathcal{S}_{H}^{0*}$ and $\mathcal{F}\mathcal{C}_{H}^{0}(0)=:\mathcal{F}\mathcal{C}_{H}^{0}$.
 Jahangiri~\cite{Jahangiri1998,Jahangiri1999} gave a sufficient condition for functions $f\in \mathcal{H}$ to be $\mathcal{FS}^{*}_{H}(\alpha)$ and $\mathcal{FC}_{H}(\alpha)$, respectively.

\begin{lemma}\label{lemFS}
Let $f=h+\overline{g}\in \mathcal{H}$, where $h$ and $g$ are given by~\eqref{phg}. Furthermore, let
$$\sum_{n=2}^{\infty}\frac{n-\alpha}{1-\alpha} |a_{n}|+\sum_{n=1}^{\infty}\frac{n+\alpha}{1-\alpha} |b_{n}|\leq 1$$
and $0 \leq \alpha < 1$. Then $f\in\mathcal{FS}^{*}_{H}(\alpha)$.
\end{lemma}

\begin{lemma}\label{lemFC}
Let $f=h+\overline{g}\in \mathcal{H}$, where $h$ and $g$ are given by~\eqref{phg}. Furthermore, let
$$\sum_{n=2}^{\infty}\frac{n(n-\alpha)}{1-\alpha} |a_{n}|+\sum_{n=1}^{\infty}\frac{n(n+\alpha)}{1-\alpha} |b_{n}|\leq 1$$
and $0 \leq \alpha < 1$. Then $f\in\mathcal{FC}_{H}(\alpha)$.
\end{lemma}

According to Rad\'{o}-Kneser-Choquet theorem, a fully convex harmonic mapping is
necessarily univalent in $\ID$. However, a fully starlike mapping need not be univalent (see~\cite{Chuaqui2004}).

For $f=h+\overline{g}\in \mathcal{S}_{H}^{0}$ we let
$$
D_f^{\epsilon}=zf_{z}-\epsilon\, \overline{z}f_{\overline{z}}\quad(|\epsilon|=1)
$$
and in the case of $\epsilon=1$,  we use the  notation $D_f^{+1}=Df$ for the sake of convenience.
As with the analytic differential operator $zf'(z)$, the operators $Df$ and $\mathscr{D}f:=D_f^{-1}$ play the important roles in harmonic mappings. In 1915, Alexander~\cite{Alexander1915}
proved that if $f\in\mathcal{S}$, then $f(z)\in \mathcal{C}$ if and only if $zf'(z)\in \mathcal{S}^{*}$.  In 1990, Sheil-Small~\cite{Sheil1990} gave one application of it to harmonic mappings as stated below.
\begin{thm}\label{thmAlex}
If $f=h+\overline{g}\in\mathcal{H}$ is univalent, and has a starlike range and if $H$ and $G$ are the analytic functions on $\ID$ defined by
%\begin{equation}\label{eqAlex}
$$zH'(z)=h(z),~ zG'(z)=-g(z),~ H(0)=G(0)=0,
$$
%\end{equation}
then $F=H+\overline{G}$ is univalent and has a convex range.
\end{thm}
In general, the converse of Theorem~\ref{thmAlex} is not necessarily true. For example, if $F=L$, where $L=H_2+\overline{G_2}\in\mathcal{C}_H^{0}$ is the harmonic right half-plane mapping defined by~\eqref{eqRL},
then it is easily see that
$$DL(z)=zH_2'(z)-\overline{z G_2'(z)}
$$
is not starlike and moreover, it is not even univalent in $\ID$ (see \cite[p.110]{Duren2004}). However, Ponnusamy and  Sairam Kaliraj~\cite{Samy2014} proved the converse theorem with an additional
condition as follows.

\begin{thm}\label{thmCAlex}
Suppose that $F = H +\overline{G}$ is a sense-preserving normalized convex mapping,
and $DF=zF_{z}-\overline{z}F_{\overline{z}}$ is sense-preserving in $\ID$. Then $DF$ is univalent and starlike in $\ID$.
\end{thm}

The above discussions show that the question of univalency of the harmonic differential operator $D_f^{\epsilon}$ is interesting.
In this paper, we study mainly the radius of univalence, fully convexity and fully starlikeness of the harmonic differential operators $D_f^{\epsilon}$ and $(1-\lambda)f+\lambda D_f^{\epsilon}$
when the coefficients of $h$ and $g$ satisfy either ~\eqref{BBHJ} or~\eqref{BBHJ2}. All these results are sharp.

%---------------------------------------------------------------------------------------------------------------%
\section{Radius of starlikeness and convexity of $D_f^{\epsilon}$}\label{Sect2}
%---------------------------------------------------------------------------------------------------------------%
In this section, we obtain the sharp radius of fully starlikeness and fully convexity of order $\alpha$ for the harmonic differential operator $D_f^{\epsilon}$.
The following identities are useful in the proof of our results:
\begin{lemma}\label{lemID}
We have
\begin{enumerate}
\item[(a)] $\displaystyle \sum_{n=2}^{\infty} n\,r^{n-1}=\frac{r(2-r)}{(1-r)^2}$,
\item[(b)] $\displaystyle\sum_{n=2}^{\infty} n^{2}r^{n-1}=\frac{r \left(4-3 r+r^2\right)}{(1-r)^3}$,
\item[(c)] $\displaystyle\sum_{n=2}^{\infty} n^{3}r^{n-1}=\frac{r \left(8-5 r+4 r^2-r^3\right)}{(1-r)^4}$,
\item[(d)] $\displaystyle\sum_{n=2}^{\infty} n^{4}r^{n-1}=\frac{r \left(16+r+11 r^2-5 r^3+r^4\right)}{(1-r)^5}$,
\item[(e)] $\displaystyle\sum_{n=2}^{\infty} n^{5}r^{n-1}=\frac{r \left(32+51 r+46 r^2-14 r^3+6 r^4-r^5\right)}{(1-r)^6}.$
\end{enumerate}
%\begin{split}
%(a). ~\sum_{n=2}^{\infty} n\,r^{n-1}&=\frac{r(2-r)}{(1-r)^2};\qquad\qquad\qquad\qquad
%(b). ~\sum_{n=2}^{\infty} n^{2}r^{n-1}=\frac{r \left(4-3 r+r^2\right)}{(1-r)^3};\\
%(c). ~\sum_{n=2}^{\infty} n^{3}r^{n-1}&=\frac{r \left(8-5 r+4 r^2-r^3\right)}{(1-r)^4}; \,\,\,\qquad
%(d). ~\sum_{n=2}^{\infty} n^{4}r^{n-1}=\frac{r \left(16+r+11 r^2-5 r^3+r^4\right)}{(1-r)^5};\\
%(e). ~\sum_{n=2}^{\infty} n^{5}r^{n-1}&=\frac{r \left(32+51 r+46 r^2-14 r^3+6 r^4-r^5\right)}{(1-r)^6}.
%\end{split}
%\end{equation*}
\end{lemma}
\begin{proof}
The proof of (a) follows from $\sum_{n=1}^{\infty} n\,r^{n}=r(1-r)^{-2}$ and the rest of them may be obtained by differentiating this and so on.
\end{proof}

For $0<r<1$ and $f\in h+\overline{g}\in \mathcal{H}$, let us define $D_f^{\epsilon,r}$ by
\begin{equation}\label{eqDfr}
D_f^{\epsilon,r}(z)=\frac{D_f^{\epsilon}(rz)}{r}=z+\sum_{n=2}^{\infty}n a_{n}r^{n-1}z^{n}
-\epsilon\,\overline{\sum_{n=2}^{\infty}n b_{n}r^{n-1}z^{n}}.
\end{equation}

\begin{theorem}\label{thmFS}
Let $f=h+\overline{g}$, where $h$ and $g$ are given by~\eqref{phg}, and the coefficients satisfy the conditions~\eqref{BBHJ}
for $n\geq 2$. Then for $D_f^{\epsilon}=zf_{z}-\epsilon\, \overline{z}f_{\overline{z}}\,\,(|\epsilon|=1)$,
\begin{itemize}
\item[(1)] the radius of fully starlikeness of order $\alpha$ is $r_{s}(\alpha)$, where $r_{s}(\alpha)$ is the unique root of the equation $p_{\alpha}(r)=0$ in the interval $(0,1)$, where
\begin{equation}\label{eqrS}
p_{\alpha}(r)=1-\alpha -(17-9 \alpha) r+(13-21 \alpha ) r^2-21(1- \alpha ) r^3+10(1-\alpha ) r^4-2(1- \alpha) r^5,
\end{equation}
\item[(2)] the radius of fully starlikeness is $r_u\approx 0.0614313$, where $r_u$ is the unique root of the equation
\begin{equation}\label{eqrS0}
1 -17 r+13 r^2-21 r^3+10 r^4-2 r^5= 0
\end{equation}
in the interval $(0,1)$.
\end{itemize}
All the results are sharp.
\end{theorem}
\begin{proof}
By the assumption $h$ and $g$ have the form~\eqref{phg} and the coefficients of the series satisfy the conditions~\eqref{BBHJ}. First, we observe that $b_{1}=g'(0)=0$. The conditions~\eqref{BBHJ} imply that the series~\eqref{phg} are convergent in $\ID$, and hence, $h$ and $g$ are analytic in $\ID$. Thus, $f=h+\overline{g}\in \mathcal{H}_0$ for $0 < r < 1$, it suffices to show that $D_f^{\epsilon,r}$ defined by~\eqref{eqDfr} with $b_1=0$ belongs to $\mathcal{F}\mathcal{S}_{H}^{0*}(\alpha)$.

Now, we consider
\begin{equation*}
\begin{split}
S_1&=\sum_{n=2}^{\infty}\frac{n-\alpha}{1-\alpha}|n a_{n}|r^{n-1}+\sum_{n=2}^{\infty}\frac{n+\alpha}{1-\alpha}|n b_{n}|r^{n-1}\\
&\leq \sum_{n=2}^{\infty}\frac{n-\alpha}{1-\alpha}\left(\frac{n(2n+1)(n+1)}{6}\right)r^{n-1}
+\sum_{n=2}^{\infty}\frac{n+\alpha}{1-\alpha}\left(\frac{n(2n-1)(n-1)}{6}\right)r^{n-1}\\
&=\sum_{n=2}^{\infty}\frac{n^2 \left(1-3 \alpha +2 n^2\right)}{3(1-\alpha) }r^{n-1}\\
&=\frac{1}{3(1-\alpha)}\left[(1-3 \alpha )\frac{r \left(4-3r+r^2\right)}{(1-r)^3}+\frac{2r \left(16+r+11 r^2-5 r^3+r^4\right)}{(1-r)^5}\right]=:T_1,
%\mbox{by Lemma~\ref{lemID} (a) and (d)}
\end{split}
\end{equation*}
where we have used Lemma~\ref{lemID}(a) and (d) in the last equality. By Lemma \ref{lemFS}, it suffices to show that $T_1\leq 1$ which holds whenever
$p_{\alpha}(r)\geq 0$, where $p_{\alpha}(r)$ is defined by~\eqref{eqrS}.

Now we shall show that the polynomial $p_{\alpha}(r)$ defined by \eqref{eqrS} has exactly one zero in the interval $(0,1)$ for every $\alpha\in[0,1)$. Since $p_{\alpha}(0)=1-\alpha>0$ and
$p_{\alpha}(1)=-16<0$, so $p_{\alpha}(r)$ has at least one zero in the interval $(0,1)$. A straightforward calculation shows that
\begin{equation*}
\begin{split}
p'_{\alpha}(r)&=9 \alpha-17+2 (13-21 \alpha ) r -63 (1-\alpha ) r^2+40 (1-\alpha ) r^3-10 (1-\alpha ) r^4\\
&=9 \alpha-17+2 (13-21 \alpha ) r- (1-\alpha ) r^2\left[23+10(r-2)^2\right]\quad\mbox{(as $0<r<1$)}\\
&< 9 \alpha-17+2 (13-21 \alpha ) r- 33(1-\alpha ) r^2=:q_{\alpha}(r).
\end{split}
\end{equation*}
It suffices to show that $q_{\alpha}(r)$ is negative in the interval $r\in(0,1)$ for every $\alpha\in[0,1)$. Moreover,
$$q'_{\alpha}(r)=2\left[(13-21\alpha)-33(1-\alpha)r\right]
$$
which gives the critical point $r_{0}=\frac{13-21\alpha}{33(1-\alpha)}$ and $q'_{\alpha}(r)>0$ for $0\leq r<r_0$ and $q'_{\alpha}(r)<0$ for $r_0<r<1$.
Clearly, we need to deal with the cases $0\leq \alpha\leq 13/21$ and
$13/21< \alpha<1$.
\begin{case}
Let $0\leq \alpha\leq 13/21$. We compute that
$$q_{\alpha}(r_{0})=\frac{8(18\alpha^2+39\alpha-49)}{33(1-\alpha)}<0\quad\mbox{for}\quad 0\leq \alpha\leq \frac{13}{21}.
$$
This implies that $q_{\alpha}(r)<0$ in the interval $r\in(0,1)$ for all $0\leq \alpha\leq 13/21$.
\end{case}
\begin{case}
Let $13/21< \alpha<1$. Then we have $q'_{\alpha}(r)<0$ for $r\in [0,1)$. Also, since $q_{\alpha}(0)=9 \alpha -17<0$ and $q_{\alpha}(1)=-24<0$, we
find that $q_{\alpha}(r)<0$ in the interval $r\in(0,1)$ and for all $13/21< \alpha<1$.
\end{case}
Combining the last two cases, we conclude that $q_{\alpha}(r)<0$ in the interval $(0,1)$ and for all $\alpha\in [0,1)$. This proves $p'_{\alpha}(r)<q_{\alpha}(r)<0$ in the interval $(0,1)$ and for all $\alpha\in [0,1)$. Thus, $D_f^{\epsilon,r}\in\mathcal{FS}^{0*}_{H}(\alpha)$ for $r\leq r_s(\alpha)$, where $r_s(\alpha)$ is the unique real root of the equation \eqref{eqrS}.

Note that the roots of the equation \eqref{eqrS} in $(0,1)$ are decreasing as a function of $\alpha$, $0\leq \alpha<1$. Consequently, $r_{s}\leq r_{u}$. Note that for $\alpha=0$, Equation~\eqref{eqrS} reduces to ~\eqref{eqrS0}. Then, by Lemma~\ref{lemFS}, we see that the harmonic function $f$ is starlike and
univalent in $|z|\leq r_{u}\approx 0.0614313$, where $r_u$ is the unique root of the equation~\eqref{eqrS0} in the interval $(0,1)$.

Next, to prove the sharpness, we consider the function
$$DF(z)=zH'(z)-\overline{zG'(z)}=H_{d}(z)+\overline{G_{d}(z)},
$$
where
$$H(z)=2z-H_{1}(z)\quad{\rm and}\quad G(z)=-G_{1}(z).
$$
Here, $H_{1}(z)$ and $G_{1}(z)$ are defined by~\eqref{eqHC}. We note that
\begin{equation*}
\begin{split}
DF(z)=z\left(2-H'_{1}(z)\right)+\overline{z G'_{1}(z)}
=z\left(2-\frac{1+z}{(1-z)^4}\right)+\overline{\left (\frac{z^2(1+z)}{(1-z)^4}\right )}.
\end{split}
\end{equation*}
Direct computation leads to
\begin{equation*}
H'_{d}(r)=2-\frac{2 r^2+5 r+1}{(1-r)^5}\quad{\rm and}\quad G'_{d}(r)=\frac{r \left(r^2+5 r+2\right)}{(1-r)^5}.
\end{equation*}
As $DF(r)$ has real coefficients, we have
\begin{equation*}
\begin{split}
J_{DF}(r)&=\left(H'_{d}(r)+G'_{d}(r)\right)\left(H'_{d}(r)-G'_{d}(r)\right)\\
&=\left(2-\frac{2 r^2+5 r+1}{(1-r)^5}+\frac{r \left(r^2+5 r+2\right)}{(1-r)^5}\right)
\left(2-\frac{2 r^2+5 r+1}{(1-r)^5}-\frac{r \left(r^2+5 r+2\right)}{(1-r)^5}\right)\\
&=\frac{\left(1 - 13 r + 23 r^2 - 19 r^3 + 10 r^4 - 2 r^5\right)\left(1 - 17 r + 13 r^2 - 21 r^3 + 10 r^4 - 2 r^5\right)}{(1-r)^5}.
\end{split}
\end{equation*}
Thus, $J_{DF}(r)=0$ in $(0,1)$ if and only if
$r=r_{u}\approx 0.0614313 $ or $r=r^{*}_{u}=0.0903331$, where $r_{u}$ and $r^{*}_{u}$ are the roots of
$$1 - 13 r + 23 r^2 - 19 r^3 + 10 r^4 - 2 r^5=0 ~\mbox{ and }~1 - 17 r + 13 r^2 - 21 r^3 + 10 r^4 - 2 r^5=0
$$
in the interval $(0,1)$, respectively. Moreover, for $r_{u}<r<r^{*}_{u}$, we have $J_{DF}(r)<0$. The graph of the function $J_{DF}(r)$ for $r\in(0,0.15)$ is shown in Figure~\ref{fR}. Therefore, in view of Lewy's theorem, the function $DF(z)$ is not univalent in $|z|<1$ if $r>r_{u}$. This shows that $r_{u}$ is sharp.

Furthermore,
\begin{equation*}
\begin{split}
\frac{\partial}{\partial\theta}\left(\arg\left(
DF(re^{i\theta})\right)\right)\bigg|_{\theta=0}&=\frac{rH'_{d}(r)-rG'_{d}(r)}{H_{d}(r)+G_{d}(r)}\\
&=\frac{1-17 r+13 r^2-21 r^3+10 r^4-2 r^5}{1-9 r+21 r^2-21 r^3+10 r^4-2 r^5}.
\end{split}
\end{equation*}
Thus, by~\eqref{eqrS} and $p_{\alpha}(r_{s}(\alpha))=0$, we have
$$\frac{\partial}{\partial\theta}\left(\arg\left(
DF(re^{i\theta})\right)\right)\bigg|_{\theta=0,r=r_{s}(\alpha)}
=\alpha.
%=\frac{1-17 r+13 r^2-21 r^3+10 r^4-2 r^5}{1-9 r+21 r^2-21 r^3+10 r^4-2 r^5}.
$$
This shows that $r_{s}(\alpha)$ is the best possible.
\end{proof}
\begin{figure}[!h]
  \centering
        \includegraphics[height=4.0in,keepaspectratio]{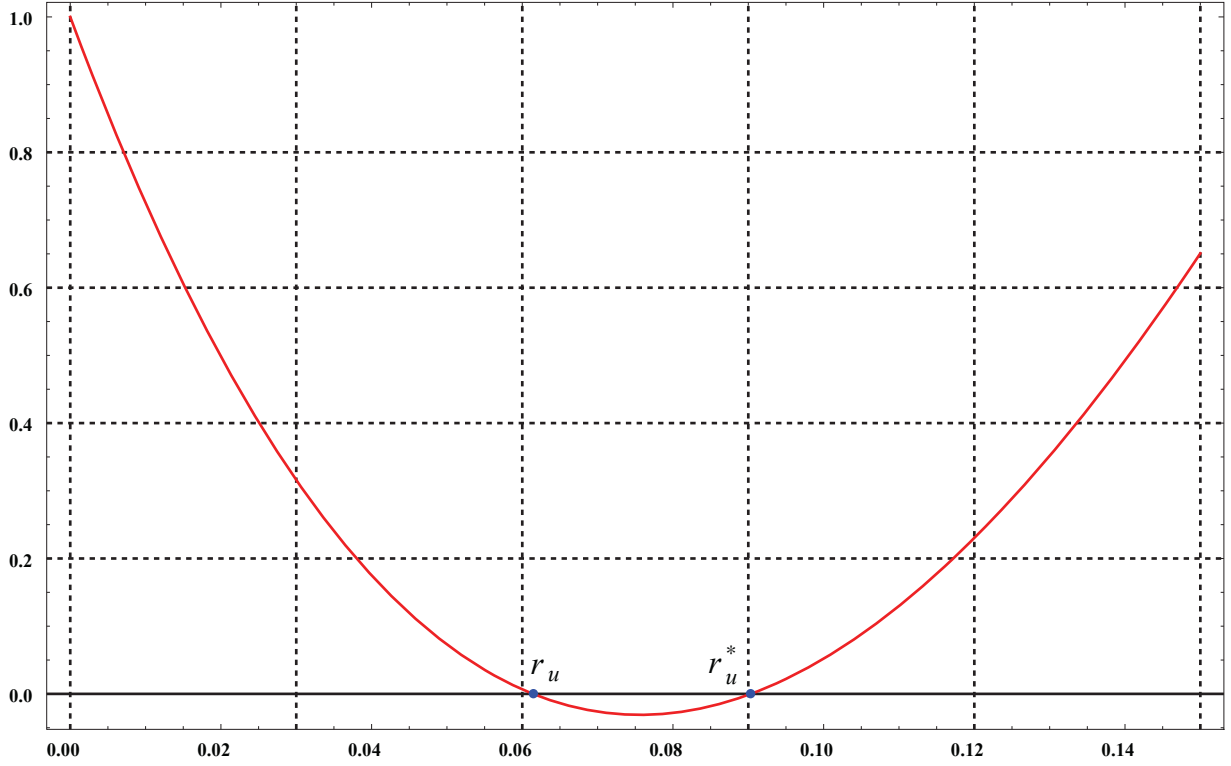}
        \caption{The roots of the Jacobian $J_{DF}(r)$ for $r\in(0,0.15)$.}\label{fR}
\end{figure}

A slight change in the proof of Theorem~\ref{thmFS} with the help of Lemma~\ref{lemFC} yields the radius of fully convex for $D_f^{\epsilon}$. So the proof is omitted here.
\begin{theorem}\label{thmFC}
Under the hypothesis of Theorem~\ref{thmFS}, for $D_f^{\epsilon}=zf_{z}-\epsilon\, \overline{z}f_{\overline{z}}\,\,(|\epsilon|=1)$ is fully convex of order $\alpha$ in $|z|<r_{c}(\alpha)$, where $r_{c}(\alpha)$ is the unique root of the equation
$$%\begin{equation}\label{eqrC}
1-\alpha -2(15-7 \alpha) r-12(1+3\alpha) r^2-2(29-21\alpha) r^3+29(1- \alpha ) r^4-12(1-\alpha) r^5+2(1- \alpha ) r^6= 0
$$%\end{equation}
in the interval $(0,1)$. In particular, $D_f^{\epsilon}$ is fully convex in $|z|<r_{c}\approx 0.0328348$, where $r_{c}$ is the unique root of the equation
$$%\begin{equation}\label{eqrC0}
1-30 r-12 r^2-58 r^3+29 r^4-12 r^5+2 r^6=0
$$%\end{equation}
in the interval $(0,1)$. All results are sharp.
\end{theorem}

\begin{theorem}\label{thmFSR}
Let $f=h+\overline{g}$, where $h$ and $g$ are given by~\eqref{phg}, and the coefficients satisfy the conditions~\eqref{BBHJ2}
for $n\geq 2$. Then for $D_f^{\epsilon}=zf_{z}-\epsilon\, \overline{z}f_{\overline{z}}\,\,(|\epsilon|=1)$,
\begin{itemize}
\item[(1)] the radius of full starlikeness of order $\alpha$ is $r_{s}(\alpha)$, where $r_{s}(\alpha)$ is the unique root of the equation $s_{\alpha}(r)=0$ in the interval $(0,1)$, where
\begin{equation}\label{eqrSR}
s_{\alpha}(r)=1-\alpha -6(2-\alpha) r+11(1-\alpha ) r^2- 8(1-\alpha) r^3+2(1-\alpha ) r^4,
\end{equation}

\item[(2)] the (univalent) radius of fully starlikeness is $r_u\approx 0.0903331$, where $r_u$ is the unique root of the equation
\begin{equation}\label{eqrSR0}
1 - 12 r + 11 r^2 - 8 r^3 + 2 r^4= 0
\end{equation}
in the interval $(0,1)$.
\end{itemize}
All the results are sharp.
\end{theorem}
\begin{proof}
%Following the notation and the method of the proof Theorem~\ref{thmFS},
We follow  the notation and the method of  proof of Theorem~\ref{thmFS}. In order to
prove (1) of Theorem~\ref{thmFSR}, it suffices to show that
$D_f^{\epsilon,r}\in\mathcal{FS}^{*}_{H}(\alpha)$, where $D_f^{\epsilon,r}(z)$ is defined by~\eqref{eqDfr}.
Accordingly, we consider
\begin{equation*}
\begin{split}
S_2&=\sum_{n=2}^{\infty}\frac{n-\alpha}{1-\alpha}|n a_{n}|r^{n-1}+\sum_{n=2}^{\infty}\frac{n+\alpha}{1-\alpha}|n b_{n}|r^{n-1}\\
&\leq \sum_{n=2}^{\infty}\frac{n-\alpha}{1-\alpha}\left(\frac{n(n+1)}{2}\right)r^{n-1}
+\sum_{n=2}^{\infty}\frac{n+\alpha}{1-\alpha}\left(\frac{n(n-1)}{2}\right)r^{n-1}\\
&=\sum_{n=2}^{\infty}\frac{n \left(n^2-\alpha\right)}{1-\alpha }r^{n-1}\\
&=\frac{1}{1-\alpha}\left[\frac{r \left(8-5 r+4 r^2-r^3\right)}{(1-r)^4}-\frac{\alpha r(2-r) }{(1-r)^2}\right]=:T_2,
\end{split}
\end{equation*}
where we have used Lemma~\ref{lemID}(a) and (c) in the last equality.
According to Lemma~\ref{lemFS}, it is sufficient to show that $T_{2}\leq 1$, which is equivalent to $s_{\alpha}(r)\geq 0$, where $s_{\alpha}(r)$ is given by~\eqref{eqrSR}.
Note that $s_{\alpha}(0)=1-\alpha>0$ and $s_{\alpha}(1)=-6<0$ and so $s_{\alpha}(r)$ has at least one zero in the interval $(0,1)$. Since
\begin{equation*}
\begin{split}
s'_{\alpha}(r)&=-2\left[3(2-\alpha)+11(1-\alpha)r-12(1-\alpha)r^2+4(1-\alpha)r^3\right]\\
&=-2\left\{3(2-\alpha)+(1-\alpha)r\left[4\left(r-\frac{3}{2}\right)^2+2\right]\right\}<0,
\end{split}
\end{equation*}
$s_{\alpha}(r)$ is strictly monotone decreasing in $(0,1)$ and hence, $s_{\alpha}(r)$ has exactly one root in $(0,1)$ for all $\alpha\in[0,1)$. Thus, $D_f^{\epsilon,r}\in\mathcal{FS}^{0*}_{H}(\alpha)$ for $r\leq r_{s}(\alpha)$, where $r_{s}(\alpha)$ is the real root of $s_{\alpha}(r)=0$ in the interval $(0,1)$.

The second part (2) of Theorem~\ref{thmFSR} follows as before by setting $\alpha=0$.

To prove the sharpness, we consider function $DF(z)=H_{d}(z)+\overline{G_{d}(z)}$ with
$$H_{d}(z)=2z-z H'_{2}(z)\quad{\rm and}\quad G_{d}(z)=-zG'_{2}(z),
$$
where $H_{2}(z)$ and $G_{2}(z)$ are defined by~\eqref{eqRL}.
Direct computation yields
\begin{equation*}
H'_{d}(r)=2-\frac{1+2r}{(1-r)^4}\quad{\rm and}\quad G'_{d}(r)=\frac{r(2+r)}{(1-r)^4}.
\end{equation*}
Again, since $DF(r)$ has real coefficients, we obtain
\begin{equation*}
\begin{split}
J_{DF}(r)&=\left(H'_{d}(r)+G'_{d}(r)\right)\left(H'_{d}(r)-G'_{d}(r)\right)\\
&=\left(2-\frac{1+2r}{(1-r)^4}+\frac{r(2+r)}{(1-r)^4}\right)
\left(2-\frac{1+2r}{(1-r)^4}-\frac{r(2+r)}{(1-r)^4}\right)\\
&=\frac{\left(1-8 r+13 r^2-8 r^3+2 r^4\right)\left(1-12 r+11 r^2-8 r^3+2 r^4\right)}{(1-r)^4}.
\end{split}
\end{equation*}
Thus, $J_{F}(r)=0$ in $(0,1)$ if and only if
$r=r_{u}\approx 0.0903331$ or $r=r^{*}_{u}=0.164878$, where $r_{u}$ and $r^{*}_{u}$ are the roots of
$$1-12 r+11 r^2-8 r^3+2 r^4=0 ~\mbox{ and } 1-8 r+13 r^2-8 r^3+2 r^4=0
$$
in the interval $(0,1)$, respectively.
The graph of the function $J_{DF}(r)$ for $r\in(0,0.25)$ is shown in Figure~\ref{fR1}. Therefore, in view of Lewy's theorem, the function $DF(z)$ is not univalent in $|z|<1$ if $r>r_{u}$. This shows that $r_{u}$ is sharp. Furthermore,
\begin{equation*}
\begin{split}
\frac{\partial}{\partial\theta}\left(\arg\left(
DF(re^{i\theta})\right)\right)\bigg|_{\theta=0}&=\frac{rH'_{d}(r)-rG'_{d}(r)}{H_{d}(r)+G_{d}(r)}\\
&=\frac{1-12 r+11 r^2-8 r^3+2 r^4}{1-6 r+11 r^2-8 r^3+2 r^4}.
\end{split}
\end{equation*}
Thus,  by~\eqref{eqrSR} and $s_{\alpha}(r_{s}(\alpha))=0$ , we have
$$\frac{\partial}{\partial\theta}\left(\arg\left(
DF(re^{i\theta})\right)\right)\bigg|_{\theta=0,r=r_{s}(\alpha)}= \alpha.
%=\frac{1-12 r+11 r^2-8 r^3+2 r^4}{1-6 r+11 r^2-8 r^3+2 r^4}
$$
This shows that $r_{s}(\alpha)$ is the best possible.
\end{proof}

\begin{figure}[!h]
  \centering
        \includegraphics[height=4.0in,keepaspectratio]{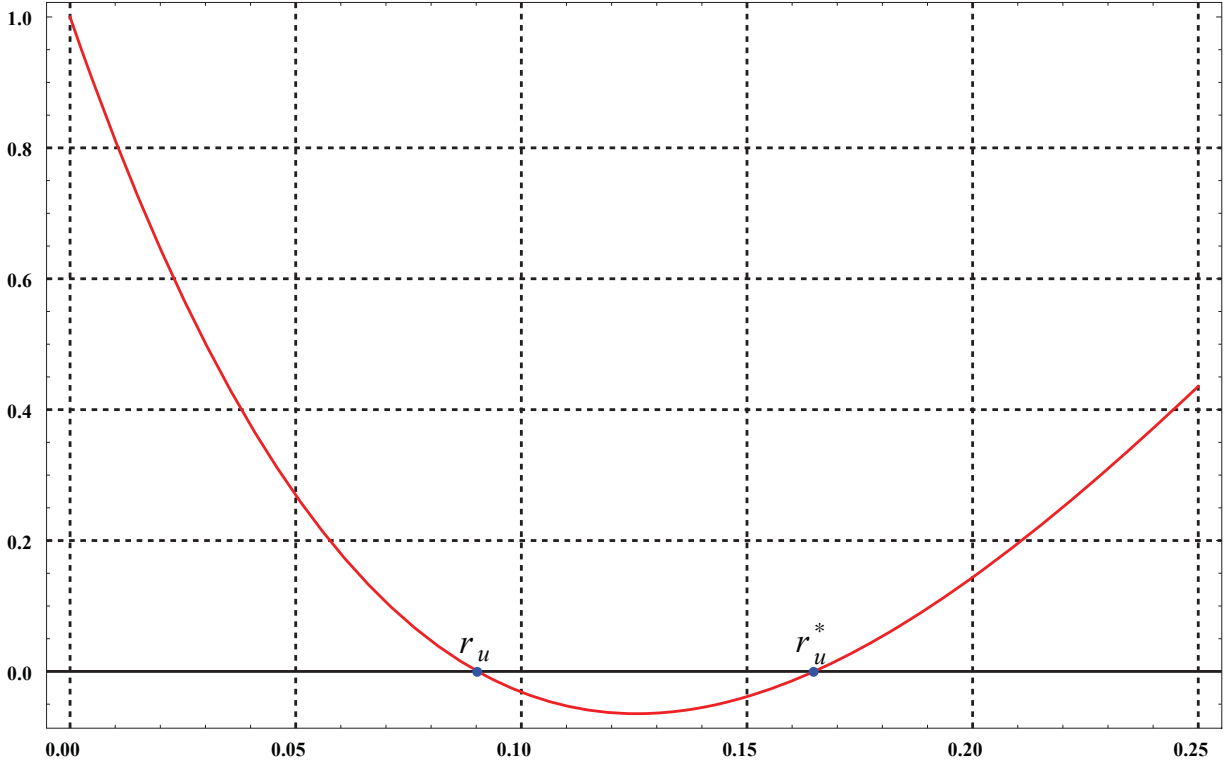}%{images/F21.pdf}
       \caption{The roots of the Jacobian $J_{DF}(r)$ for $r\in(0,0.25)$.}\label{fR1}
\end{figure}

The following result can be obtained by a similar argument as in the proof of Theorem~\ref{thmFSR}, and so we omit the proof here.
\begin{theorem}\label{thmFCR}
Let $f=h+\overline{g}$, where $h$ and $g$ are given by~\eqref{phg}, and the coefficients satisfy the conditions~\eqref{BBHJ2}
for $n\geq 2$. Then $D_f^{\epsilon}=zf_{z}-\epsilon\, \overline{z}f_{\overline{z}}\,\,(|\epsilon|=1)$ is fully convex of order $\alpha$ in $|z|<r_{c}(\alpha)$, where $r_{c}(\alpha)$ is the unique root of the equation
$$%\begin{equation}\label{eqrCC}
1-\alpha -(23-10 \alpha) r+(19-30 \alpha ) r^2-(41-42 \alpha) r^3+(30-31 \alpha ) r^4-12(1-\alpha) r^5+2(1-\alpha ) r^6= 0
$$%\end{equation}
in the interval $(0,1)$. In particularly, $D_f^{\epsilon}$ is fully convex in $|z|<r_{c}\approx 0.0449935$, where $r_{c}$ is the unique root of the equation
$$%\begin{equation}\label{eqrC1}
1-23 r+19 r^2-41 r^3+30 r^4-12 r^5+2 r^6=0
$$%\end{equation}
in the interval $(0,1)$. All results are sharp.
\end{theorem}

\section{Radius of the harmonic linear differential operator $(1-\lambda)f+\lambda D_f^{\epsilon}$}\label{Sect3}

For $\lambda\geq 0$, let us introduce
$$\mathcal{K}_{H}^{1}(\lambda)=\left\{f=h+\overline{g}\in \mathcal{H}: \RE \left(h'_{\lambda}(z)\right)>|g'_{\lambda}(z)|,~ z\in\ID\right\},
$$
where $h_{\lambda}(z)=(h*\phi_{\lambda})(z)$ and $g_{\lambda}(z)=(g*\phi_{\lambda})(z),$ with
$$\phi_{\lambda}(z)=(1-\lambda)\frac{z}{1-z}+\lambda \frac{z}{(1-z)^2}=\sum_{n=1}^{\infty}\left(1-\lambda+n\lambda\right)z^n.$$
Similarly, we define
$$\mathcal{K}_{H}^2(\lambda)=\left\{f=h+\overline{g}\in \mathcal{H}: | h'_{\lambda}(z)-1|<1-|g'_{\lambda}(z)|, ~z\in\ID\right\}.
$$
Also, we note that $\mathcal{K}_{H}^{1}:=\mathcal{K}_{H}^{1}(0)$ and $\mathcal{K}_{H}^{2}:=\mathcal{K}_{H}^{2}(0)$.
Next, we observe that if $f\in \mathcal{K}_{H}^1(\lambda)$ and $F_{\lambda}(z)=h_{\lambda}(z)+\epsilon g_{\lambda}(z)$, then
$$\RE \left(F'_{\lambda}(z)\right)
=\RE\left(h'_{\lambda}(z) + \epsilon g'_{\lambda}(z)\right)\geq  \RE\left(h'_{\lambda}(z)\right)-|g'_{\lambda}(z)|>0
$$
for each $|\epsilon| \leq 1$. Moreover, because $F_{\lambda}(z)=((h+\epsilon g)*\phi_{\lambda})(z)$, we see that
\begin{equation*}
\begin{split}
F'_{\lambda}(z)&=\left(h'(z)+\epsilon g'(z)\right)*\left(\frac{\phi_{\lambda}(z)}{z}\right)\\
&=\left(h'(z)+\epsilon g'(z)\right)*\left(\sum_{n=1}^{\infty}\left(1-\lambda+\lambda n\right)z^{n-1}\right)
\end{split}
\end{equation*}
so that
$$F'_{\lambda}(z)*\frac{\Psi_{\lambda}(z)}{z}=h'(z)+\epsilon g'(z), \quad \Psi_{\lambda}(z) =\sum_{n=1}^{\infty}
\frac{z^{n}}{1-\lambda+\lambda n}.
$$
By the convolution result, since $\Psi_{\lambda}$ is convex in $\ID$ and thus, ${\rm Re}\,(\Psi_{\lambda}(z)/z)>1/2$ in $\ID$,
it follows that ${\rm Re}\,(h'(z)+\epsilon g'(z))>0$ in $\ID$ and therefore,
$h+\epsilon g$ is close-to-convex for each $|\epsilon|\leq 1$. Consequently, by Theorem~\ref{thmClinie}, the harmonic function $f=h+\overline{g}$ is close-to-convex in $\ID$ whenever $f\in \mathcal{K}_{H}^1(\lambda)$.
Thus, functions in $\mathcal{K}_{H}^1(\lambda)$ are close-to-convex in $\ID$.

%The following result is basic and is obvious from~\cite{Kalaj2014}. So, we omit its proof.
The following result follows immediately from~\cite{Kalaj2014}, and we omit its proof.
\begin{theorem}\label{thmK2}
Let $f=h+\overline{g}$, where $h$ and $g$ have the form~\eqref{phg} with $J_f(0)=1-|b_1|^2>0$. If
\begin{equation}\label{eqC1}
\sum_{n=2}^{\infty}(1-\lambda+\lambda n) n \left(|a_{n}|+|b_{n}|\right)\leq 1-|b_1|
\quad(\lambda\geq 0).
\end{equation}
holds, then $f\in \mathcal{K}_H^2(\lambda)$.
\end{theorem}

\begin{theorem}\label{thmSt}
Let $h$ and $g$ have the form~\eqref{phg} and the coefficients of the series satisfy the conditions~\eqref{BBHJ}.
Then, $f=h+\overline{g}$ satisfies the inequality
\begin{equation*}%\label{HUP1}
|h'_{\lambda}(z)-1|<1-|g'_{\lambda}(z)|\quad(\lambda\geq 0)
\end{equation*}
in the disk $|z|<r_{s}$ and is fully starlike in $|z|<r_{s}$, where $r_{s}$ is the root of the equation
\begin{equation}\label{eqDOSR}
1-(11+6 \lambda ) r+(21-8 \lambda ) r^2-(19+2 \lambda ) r^3+10 r^4-2 r^5=0
\end{equation}
in the interval $(0,1)$. The result is sharp.
\end{theorem}
\begin{proof}
By assumption~\eqref{BBHJ}, $f = h + \overline{g}$ is harmonic in $\ID$. Let $0 < r < 1$. It suffices to show that
the coefficients of $f_r(z)=r^{-1}f(rz)$ satisfy the inequality~\eqref{eqC1}, where
\begin{equation}\label{eqDFr0}
f_{r}(z)=\frac{f(rz)}{r}=z+\sum_{n=2}^{\infty} a_{n}r^{n-1}z^{n}
+\overline{\sum_{n=2}^{\infty} b_{n}r^{n-1}z^{n}}.
\end{equation}
By hypotheses
$$n\left(|a_n|+|b_n|\right)\leq\frac{n(2n^2+1)}{3}$$
and thus, we compute
\begin{equation*}
\begin{split}
S_{3}&=\sum_{n=2}^{\infty}(1-\lambda+\lambda n) n (|a_{n}|+|b_{n}|)r^{n-1}\\
&\leq \sum_{n=2}^{\infty}(1-\lambda+\lambda n)\frac{n(2n^2+1)}{3}r^{n-1}\\
&=\frac{1}{3}\sum_{n=2}^{\infty}\left[2\lambda n^4 +2(1-\lambda)n^3 +\lambda n^2 +(1-\lambda)n\right]r^{n-1}=:T_3,
\end{split}
\end{equation*}
which is less than or equal to 1 if $T_3\leq 1$, which is equivalent to
$$1-(11+6 \lambda ) r+(21-8 \lambda ) r^2-(19+2 \lambda ) r^3+10 r^4-2 r^5\geq 0.$$
Thus, from Theorem~\ref{thmK2}, $f_{r}(z)=r^{-1}f(rz)$ is close-to-convex (univalent) and fully starlike in $\ID$ for all $0<r\leq r_{s}$, where $r_{s}$ is the root of the equation~\eqref{eqDOSR} in the interval $(0,1)$. In particular, $f$ is close-to-convex (univalent) and fully starlike in $|z|<r_{s}$.

The proof of sharpness is similar to Theorem~\ref{thmFS} and so we omit it here.
\end{proof}

\begin{remark}
If $\lambda=0$ in Theorem~\ref{thmSt}, the equation~\eqref{eqDOSR} reduces to
$$(1-r)\left(1-10r+11r^2-8r^3+2r^4\right)=0.$$
Thus, $f_{r}(z)\in \mathcal{FS}^{*}_{H}\cap \mathcal{K}_{H}^2$ for $r\leq r_{S}\approx 0.112903$, where $r_{S}$ is the unique root of the equation
$$1-10r+11r^2-8r^3+2r^4=0.$$
This gives the result obtained in \cite[Lemma 1.2]{Kalaj2014}.

If $\lambda=1$ in Theorem~\ref{thmSt}, then the equation~\eqref{eqDOSR} reduces to the equation~\eqref{eqrS0} in Theorem~\ref{thmFS}. Thus, the univalent radius is $r_u\approx 0.0614313$, where $r_u$ is the root of the equation~\eqref{eqrS0}.
\end{remark}

\begin{corollary}\label{clS}
Let $f=h+\overline{g}$, where $h$ and $g$ are given by~\eqref{phg}, and the coefficients satisfy the conditions~\eqref{BBHJ}
for $n\geq 2$. Then the radius of fully starlikeness for $F(z)=(1-\lambda)f+\lambda D_f^{\epsilon}\,\,(0\leq\lambda\leq 1)$ is at least $r_{s}$, where $r_{s}$ is the root of the equation~\eqref{eqDOSR} in the interval $(0,1)$.
\end{corollary}
\begin{proof}
If we write
$$F(z)=\frac{F(rz)}{r}=z+\sum_{n=2}^{\infty}A_n z^n+\overline{\sum_{n=2}^{\infty}B_n z^n},
$$
then
$$A_n=\left[(1-\lambda)+\lambda n\right] a_{n}r^{n-1}\quad{\rm and}\quad
B_n=\left[(1-\lambda)+\overline{\epsilon}\,\lambda n\right] b_{n}r^{n-1}~(\lambda>0)
$$
so that
$$n\left(|A_n|+|B_n|\right)\leq (1-\lambda+\lambda n)\left(|a_n|+|b_n|\right)n r^{n-1}
\leq (1-\lambda+\lambda n)\frac{n(2n^2+1)}{3} r^{n-1}.$$

Thus, we desired result follows from the proof of Theorem~\ref{thmSt}.
\end{proof}

\begin{theorem}\label{thmCt}
Let $h$ and $g$ have the form~\eqref{phg} and the coefficients of the series satisfy the conditions~\eqref{BBHJ2}.
Then, $f=h+\overline{g}\in K_{H}^2(\lambda)~(\lambda\geq 0)$
in the disk $|z|<r_{c}$ and is fully starlike in $|z|<r_{c}$, where $r_{c}$ is the root of the equation
\begin{equation}\label{eqDOSR2}
1-4(2+\lambda) r+(13-2 \lambda ) r^2-8 r^3+2 r^4=0
\end{equation}
in the interval $(0,1)$. The result is sharp.
\end{theorem}
\begin{proof}
Clearly, it suffices to observe that
$$n\left(|a_n|+|b_n|\right)\leq n\left(\frac{n+1}{2}+\frac{n-1}{2}\right)=n^2
$$
and thus,
$$S_{4}=\sum_{n=2}^{\infty}(1-\lambda+\lambda n) n (|a_{n}|+|b_{n}|)r^{n-1}
\leq (1-\lambda)\sum_{n=2}^{\infty}n^2 r^{n-1}+\lambda\sum_{n=2}^{\infty}n^3 r^{n-1}=:T_4.
$$
It follows that $S_4\leq 1$ if $T_4\leq 1$ which is equivalent to
$$1-4(2+\lambda) r+(13-2 \lambda ) r^2-8 r^3+2 r^4\geq 0.
$$
The desired conclusion  and the sharpness follow as before.
\end{proof}

If $\lambda=0$ in Theorem~\ref{thmCt}, the equation~\eqref{eqDOSR2} reduces to
$$(1-r)\left(1-7 r+6 r^2-2 r^3\right)= 0
$$
and this gives the following result obtained in \cite{Kalaj2014}.

\begin{thm}\label{thDS2}
Let $h$ and $g$ have the form~\eqref{phg} and the coefficients of the series satisfy the conditions~\eqref{BBHJ2}. Then, $f=h+\overline{g}$ satisfies the inequality
\begin{equation*}
|h'(z)-1|<1-|g'(z)|
\end{equation*}
in the disk $|z|<r_{S}\approx 0.164878$ and fully starlike in $|z|<r_{S}$, where $r_{S}$  is the root of the equation
\begin{equation*}%\label{eqDS2}
1-7r+6r^2-2r^3=0
\end{equation*}
in the interval $(0,1)$. The result is sharp.
\end{thm}

If $\lambda=1$ in Theorem~\ref{thmCt}, the equation~\eqref{eqDOSR2} reduces to the equation~\eqref{eqrSR0} in Theorem~\ref{thmFSR}(2).
Moreover, by Lemma~\ref{lemFS} and Theorem~\ref{thmCt}, we can easily obtain the following result.

\begin{corollary}\label{clS1}
Let $f=h+\overline{g}$, where $h$ and $g$ are given by~\eqref{phg}, and the coefficients satisfy the conditions~\eqref{BBHJ2}
for $n\geq 2$. Then the radius of fully starlikeness for $F(z)=(1-\lambda)f+\lambda D_f^{\epsilon}$ is at least $r_{c}$. where $r_{c}$ is the root of the equation~\eqref{eqDOSR2} in the interval $(0,1)$. The result is sharp.
\end{corollary}
\begin{theorem}\label{clC}
Let $h$ and $g$ have the form~\eqref{phg} with $|b_{1}|=|g'(0)|<1$, and the coefficients satisfy the conditions
$$|a_{n}|+|b_{n}|\leq c $$
for $n\geq 2$. Then, $f=h+\overline{g}\in \mathcal{K}_H^2(\lambda)$ and
is fully starlike in $|z|<r_{v}$, where $r_{v}$ is the root of the equation $\Phi_{c,|b_1|,\lambda}(r)=0$ in the interval $(0,1)$, where
\begin{equation}\label{eqRC2}
\Phi_{c,|b_1|,\lambda}(r)=(1+c-|b_{1}|)(1-r)^3-c\left[1+(2\lambda-1)r\right].
\end{equation}
 The result is sharp.
\end{theorem}
\begin{proof}

We apply Theorem~\ref{thmK2} and show that $f_{r}$ defined by~\eqref{eqDFr0} belongs to $\mathcal{K}_{H}^{2}(\lambda)$.
As in the proofs of Theorems~\ref{thmSt} and \ref{thmCt}, it is suffices to show that the corresponding coefficient
inequality~\eqref{eqC1}, namely,
\begin{equation*}
\begin{split}
S_5&=\sum_{n=2}^{\infty}(1-\lambda+\lambda n) n (|a_{n}|+|b_{n}|)r^{n-1}+|b_1|\\
&\leq \sum_{n=2}^{\infty}(1-\lambda+\lambda n) n c\, r^{n-1}+|b_1|\leq 1.
\end{split}
\end{equation*}
By Lemma~\ref{lemID} $(a)$ and $(b)$, the last inequality is easily seen to be equivalent to
$$c\left[(1-\lambda)\frac{r(2-r)}{(1-r)^2}+\lambda\frac{r \left(4-3 r+r^2\right)}{(1-r)^3}\right]\leq 1-|b_1|$$
which upon simplification reduces to $\Phi_{c,|b_1|,\lambda}(r)\geq 0$. The result follows.
$$(1+c-|b_{1}|)(1-r)^3-c\left[1+(2\lambda-1)r\right]\geq 0.
$$

The function $f_0=h_0+\overline{g_0}=(1-\lambda)h_1+\lambda zh'_1+\overline{(1-\lambda)g_1+\lambda zg'_1}$, where
$$h_1(z)=z-\frac{c}{2}\left(\frac{z^2}{1-z}\right)
\quad{\rm and}\quad
g_1(z)=-|b_1|z-\frac{c}{2}\left(\frac{z^2}{1-z}\right),
$$
shows that the result is sharp. Note that
\begin{equation*}
J_{f_0}(r)=|h'_{0}(r)|^2-|g'_{0}(r)|^2
=\frac{(1+|b_1|)}{(1-r)^3}\left((1+c-|b_{1}|)(1-r)^3-c\left[1+(2\lambda-1)r\right]\right)
\end{equation*}
which shows that $J_{f_0}(r)>0$ for $r<r_v$. The proof of the theorem is complete.
\end{proof}

\begin{remark}
Take $\lambda=0$ in Theorem~\ref{clC}, the equation~\eqref{eqRC2} reduces to
$$(1-r)\left[(1+c-|b_{1}|)(1-r)^2-c\right]=0.$$
Then $f_{r}(z)\in \mathcal{FS}_{H}^{*}$ for $r<r_{S}=1-\sqrt{\frac{c}{1+c-|b_{1}|}}$.
This result is due to Kalaj et al.~\cite[Lemma 1.6]{Kalaj2014}.
\end{remark}

%\newpage{}
%---------------------------------------------------------------------------------------%
 \vskip.20in
%\begin{center}{\sc  Conflict of interests}\end{center}

%\vskip.05in
%The authors declare that there is no conflict of
%interests regarding the publication of this article.

%---------------------------------------------------------------------------------------%
\begin{center}{\sc Acknowledgements}
\end{center}

\vskip.05in
The research of the first author was supported by the
First Batch of Young and Middle-aged Academic Training Object Backbone of Honghe University under Grant No. 2014GG0102. The work of the first author
was completed during his visit to the Indian Statistical Institute (ISI), Chennai Centre.
The second author is currently at the ISI, Chennai Centre.
%The authors would like to thank the referees for their valuable suggestions and comments which essentially improved the quality of this paper.

%---------------------------------------------------------------------------------------%

%---------------------------------------------------------------------------------------%
%{\bf Acknowledgements:}

%\bibliographystyle{amsplain}

\end{document}